\documentclass [11pt]{amsart}
\usepackage[top=1.2in, bottom=1.2in, left=1in, right=1in]{geometry}
\usepackage{xcolor}
\usepackage{amssymb,amsmath,amsthm,amsfonts}
\usepackage{enumerate}
\usepackage{setspace}
\usepackage{empheq}
\usepackage[all]{xy}
\usepackage{verbatim}
\usepackage[normalem]{ulem}
\usepackage{todonotes}
\usepackage[bookmarks,colorlinks,breaklinks]{hyperref}  
\hypersetup{linkcolor=blue,citecolor=blue,filecolor=blue,urlcolor=blue} 
\usepackage{graphicx}
\usepackage{float}

\numberwithin{equation}{section}
\numberwithin{figure}{section}

\newtheorem{theorem}{Theorem}[section]
\newtheorem{proposition}[theorem]{Proposition}
\newtheorem{lemma}[theorem]{Lemma}
\newtheorem{remark}[theorem]{Remark}

\newtheorem{question}[theorem]{Question}

\theoremstyle{definition}
\newtheorem{definition}[theorem]{Definition}
\newtheorem{example}[theorem]{Example}

\makeatletter

\newcommand{\Rmnum}[1]{\expandafter\@slowromancap\romannumeral #1@}
\makeatother

\newcommand{\MCG}{\text{MCG}}

\newcommand*{\defeq}{\mathrel{\vcenter{\baselineskip0.5ex \lineskiplimit0pt
		\hbox{\scriptsize.}\hbox{\scriptsize.}}}%
=}

\begin{document}
\title{Relative train tracks and generalized endperiodic graph maps}

\author{Yan Mary He}
\address{Department of Mathematics, University of Oklahoma, Norman, OK 73019}
\email{he@ou.edu}

\author{Chenxi Wu}
\address{Department of Mathematics, University of Wisconsin-Madison, Madison,
WI 53703}
\email{cwu367@wisc.edu}

\date{\today}

\begin{abstract}
Motivated by the work of Cantwell-Conlon-Fenley on endperiodic homeomorphisms of infinite type surfaces,
we define and study endperiodic and generalized endperiodic maps of an infinite graph with finitely many ends. Adapting the work of Bestvina-Handel to the infinite type setting, we define endperiodic relative train track maps.
We prove that any generalized endperiodic map is homotopic to a generalized endperiodic relative train track map, via a combinatorially bounded homotopy equivalence. We show that the (largest) Perron-Frobenius eigenvalue of a relative train track representation of a generalized endperiodic map $f$ is a canonical quantity associated to $f$ as it admits a canonical group theoretic interpretation.
Moreover, the (largest) Perron-Frobenius eigenvalue and the topological entropy of a relative train track map is the smallest among its proper homotopy equivalence class.

\end{abstract}

\maketitle

\section{Introduction}

\subsection{Motivation and background}
Let $G$ be a connected and locally finite graph. A graph map $f \colon G \to G$ is a homotopy equivalence. If $G$ is a finite graph (i.e., $G$ has finitely many vertices and edges), graph maps of $G$ are extensively studied as such a map can be viewed as a topological representative of an outer automorphism of the free group $\pi_1(G)$.
The outer automorphism group ${\rm Out}(F_n)$ of a rank $n\ge 2$ free group $F_n$ is a counterpart in real dimension one of the mapping class group $\MCG(S)$ of a compact surface $S$, as $\MCG(S) = {\rm Out}(\pi_1(S))$ by the Dehn-Nielsen-Baer theorem. Moreover, ${\rm Out}(F_n)$ has been proven to share many similar properties as the mapping class group of a compact surface; we refer the reader to the survey paper \cite{Bestvina02}.

If $G$ is an infinite graph, the structure of a graph map is more complicated and much less is known. However, in real dimension two, the study of homeomorphisms of infinite type surfaces 
and big mapping class groups has become an active field of research. In a blogpost \cite{Calegari}, Calegari proposed the study of the mapping class group ${\rm MCG}(\mathbb R^2 \setminus \mathcal C)$ of the plane minus a Cantor set $\mathcal C$. More specifically, he posed the question of whether the space of quasimorphisms of this group is infinite-dimensional, which is the case for mapping class groups of finite type surfaces \cite{BF02}. To attack the problem, Calegari defined the {\it ray graph} of the plane minus the Cantor set, which is an analogue of the curve complex for finite type surfaces. Bavard proved in her thesis \cite{Bavard16} that the ray graph has infinite diameter, the action of ${\rm MCG}(\mathbb R^2 \setminus \mathcal C)$ on the ray graph is hyperbolic and the  space of quasimorphisms of ${\rm MCG}(\mathbb R^2 \setminus \mathcal C)$ is infinite dimensional. Ever since, the field has seen rapid development. We refer the interested reader to the survey paper \cite{AV20} for recent results and development of mapping class groups of infinite type surfaces.

More recently, in \cite{CCF21}, Cantwell, Conlon and Fenley introduced and studied a class of homeomorphisms of infinite type surfaces called {\it endperiodic maps}. If $S$ is a compact surface with negative Euler characteristic, the Nielsen-Thurston theory classifies the isotopy class of a homeomorphism $f$ of $S$; see \cite{CassonBleiler88,FarbMargalit12}. For endperiodic homeomorphisms of infinite type surfaces with finitely many ends, Cantwell, Conlon and Fenley systematically developed an analogous theory which Handel and Miller had outlined but not published. Both theories show that $f$ is isotopic to a map $h$ which preserves a pair of transverse geodesic laminations on $S$. For compact surfaces, $h$ preserves a finite disjoint collection of simple closed curves on $S$ which decompose $S$ into periodic pieces and pseudo-Anosov pieces. Similarly, for endperiodic maps, $h$ preserves a disjoint collection of simple closed curves or lines on $S$ which decompose $S$ into finitely many $h$-orbits of compact surfaces, finitely many non-compact pieces on which a power of $h$ is a translation, and finitely many non-compact pseudo-Anosov pieces. Endperiodic surface maps arise naturally in three-dimensional hyperbolic geometry. For example, they can be seen as limit points in Thurston's fibered cone \cite{thurston1986norm}, and they are related to depth 1 foliations on closed 3-manifolds \cite{fenley1997end, landry2023endperiodic}.

\subsection{Statement of results}
Inspired by the work of Cantwell, Conlon and Fenley \cite{CCF21}, in this paper, we define analogously {\it endperiodic maps} and {\it generalized endperiodic maps} of an infinite graph and study their dynamical properties. Roughly speaking, a generalized endperiodic map is a cellular map from a graph to itself which is a proper homotopy equivalence (as defined in \cite{AKB}), such that a power of it fixes all the ends, and every end has a neighborhood whose forward or backward image under this power is contained in itself. An endperiodic map is a generalized endperiodic map which is a homeomorphism near the ends. See Definition \ref{def_endper} for precise definitions. We remark that our definition of an endperidoic map is 
the same as the one given by Meadow-MacLeod \cite{meadow2024end}.

For a finite type surface $S$, the Nielsen-Thurston theory provides a normal form for mapping classes of $S$. Analogously, Bestvina and Handel \cite{BH92} introduced {\it relative train track maps} as a ``normal form'' for an outer automorphisms of a finitely generated free group $F_n$.  For irreducible outer automorphisms of $F_n$, relative train track maps are the ``most efficient'' topological representatives.

The first goal of this paper is to adapt Bestvina-Handel's definition of relative train tracks to the infinite type setting for endperiodic maps. The precise definition is given in Definition \ref{def_reltt}.
Our first theorem gives a normal form for a generalized endperiodic graph map.

\begin{theorem}\label{thm_traintrack}
Let $f \colon G\to G$ be a generalized endperiodic (resp. endperiodic) map of an infinite graph $G$ with finitely many ends. Then $f \colon G \to G$ is conjugate to a generalized endperiodic (resp. endperiodic) relative train track map via a combinatorially bounded homotopy equivalence. Furthermore, if the relative train track map is endperiodic with only a single exponentially growing finite strata, then it can be made into a train track map.
\end{theorem}

See Definition \ref{def:cbhe} and Definition \ref{def_tranmatrix} for the terms combinatorially bounded homotopy equivalence and exponentially growing, respectively.
We note that the combinatorially bounded homotopy equivalences in our case are also proper homotopy equivalences in the sense of \cite{AKB}.

\medskip

The dynamics of a relative train track map are encoded in the set of {\it stretch factors}. Associated to a  relative train track map $F \colon G \to G$, there is an invariant filtration $G_0\subseteq G_1\subseteq G_2\dots$ of the graph $G$. On each strata $H_i \defeq \overline{G_i\backslash G_{i-1}}$, the induced map has a stretch factor $\lambda_i \in \mathbb{R}$, which measures the asymptotic growth rate of a generic loop in $H_i$. We denote by $\lambda(F) \defeq \max_i\lambda_i$ the maximal stretch factor.

Let $f \colon G \to G$ be a generalized endperiodic map and let $F \colon \mathcal{G} \to \mathcal{G}$ be a generalized endperiodic relative train track representation of $f$.
Our next theorem gives a group theoretic interpretation of the dynamical quantity $\lambda(F)$. This implies that $\lambda(F)$ is a canonical quantity associated to the generalized endperiodic map $f$.

\begin{theorem}\label{thm_1.2}
Let $f \colon G \to G$ be a generalized endperiodic map, $F$ its generalized endperiodic relative train track representation, and $\lambda =\lambda(F) =\max_i \lambda_i(F)$. Then the following statements hold.
\begin{enumerate}
\item If $\lambda=0$, then for every loop $\gamma$ on $G$, any compact subgraph $G'$, there is some $n>0$ such that $f^{n'}(\gamma)$ is homotopic to a loop disjoint from $G'$ for all $n'>n$.
\item If $\lambda>1$, then
\[\log(\lambda)=\sup_{G'\subseteq G\text{ finite},~ \mathfrak{c} \text{ conjugacy class of }\pi_1(G)}\limsup_{n\rightarrow\infty}\frac{\log(l_w(P_{G'}(f_*^n(\mathfrak{c})))}{n}.\]
Moreover, there exist a finite subgraph $G'$ and a conjugacy class $\mathfrak{c}$ which realize the supremum.
\end{enumerate}
\end{theorem}

The term $l_w(P_{G'}(\cdot))$ is defined as follows. Consider a finite subgraph $G'$ of $G$. Collapsing all edges in its complement $G \setminus G'$, we get a finite graph $G_1$ and a group homomorphism $P_{G'} \colon \pi_1(G)\rightarrow \pi_1(G_1)$. We denote by $l_w$ the word length on the finitely generated free group $\pi_1(G_1)$.

\begin{remark} \label{rmk_1.3}
A generalized endperiodic map $f \colon G \to G$ may have several relative train track representations. However, the maximal stretch factor is the same for any relative train track representation.
Indeed, if $F$ is a generalized endperiodic relative train track representation, by Theorem \ref{thm_1.2}, its maximal stretch factor $\lambda(F)$ can be expressed in terms of conjugacy classes of fundamental groups of finite subgraphs of $G$.
Therefore, the maximal stretch factor $\lambda(F)$ is a canonical quantity determined by the map $f$, and we sometimes denote it by $\lambda(f)$.
\end{remark}

\medskip

Theorem \ref{thm_traintrack} and Theorem \ref{thm_1.2} imply that generalized endperiodic relative train track representations of $f\colon G \to G$ have the smallest maximal stretch factor in the {\it combinatorially bounded} homotopy class of $f$. In fact, a stronger statement is true.
We prove in the next theorem that relative train track representations have the smallest maximal stretch factor in the proper homotopy class of $f$.

\begin{theorem} \label{thm_1.4}
Let $f_1 \colon G_1 \to G_1$ and $f_2 \colon G_2 \to G_2$ be generalized endperiodic map that are proper homotopy equivalent to each other. Let $F_i \colon G'_i \to G'_i$ be a relative train track representation of $f_i, i=1,2$. Then we have $\lambda(F_1) = \lambda(F_2)$. If $\lambda(F_1) \ge 1$, {\color{black} then $\log(\lambda(F_1)) = h_{\rm top}(F_1)$, the topological entropy of the induced map of $F_1$ on the one point compactification of $G_1'$.
Furthermore, if $f_3 \colon G_3 \to G_3$ is a generalized endperiodic map which is proper homotopy equivalent to $F_1$, then we have $\log(\lambda(F_1)) \le h_{\rm top}(f_3)$, 
the topological entropy of the induced map of $f_3$ on the one point compactification of $G_3$.}

\end{theorem}

Before we move on to discuss the ideas of the proofs, we mention some other related works on graph maps and mapping class groups of infinite graphs. In \cite{AKB}, Algom-Kfir and Bestvina proposed a definition of the mapping class group of an infinite graph, an ``infinite-type'' analogue of ${\rm Out}(F_n)$, and proved the Nielsen realization theorem. In \cite{DHK23,DHK23b}, Domat, Hoganson and Kwak studied the coarse geometry of pure mapping class groups of infinite graphs.  In \cite{meadow2024end}, Meadow-MacLeod considered endperiodic maps and endperiodic train track maps of an infinite graph. While our definition of endperiodic maps coincides with Meadow-MacLeod's, our concept of relative train track maps is slightly different from Meadow-MacLeod's concept of endperiodic train track maps, as we do not have any constraints on a relative train track map on the {\it non-exponentially growing components}.

While our methods work for generalized endperiodic maps, in the case of endperiodic maps, it is possible to modify our relative train tracks into the ones as defined in \cite{meadow2024end}, which would allow one to study their mapping torus and obtain results analogous to \cite{fenley1997end, landry2023endperiodic}.

\subsection{Strategy of the proofs}
We prove Theorem \ref{thm_traintrack} by adapting the Bestvina-Handel algorithm \cite[Section 5]{BH92}. Building upon previous works, especially Stallings' foldings \cite{Sta83}, Bestvina and Handel introduced {\it moves} that can be applied to a graph map $f \colon G \to G$ of a finite graph to make it efficient. Here a graph map is {\it efficient} if the image of every edge is a path which does not have backtracks. If $f$ is irreducible, namely, $f$ does not preserve any non-trivial proper subgraph of $G$, then the efficiency of $f$ is measured by the (Perron-Frobenius) eigenvalue of its transition matrix. For irreducible graph maps, Bestvina and Handel gave an algorithm \cite[Section 1]{BH92} which takes an irreducible graph map as an input and outputs an efficient irreducible graph map. The algorithm consists of applying the moves in a certain way such that it stops when the graph map is made efficient. The key point of the validity of the algorithm is that it eventually stops. This is achieved by decreasing (or not increasing) the Perron-Frobenius eigenvalue of the transition matrix after applying a move at each step of the algorithm, and the possible values for the eigenvalue are finite.
For general graph maps that are not necessarily irreducible, Bestvina and Handel generalized the algorithm in \cite[Section 5]{BH92} by considering a filtration of $G$ determined by the dynamics of $f$.

An endperiodic graph map $f \colon G \to G$ always preserves a subgraph due to the existence of the attracting ends. However, the behavior of $f$ near an end of $G$ is sufficiently simple so that there is still a stratification of $G$ based on the dynamics of $f$. We follow the general framework of the Bestvina-Handel algorithm, except that we need to take care of the ends of the graph. In particular, we apply the moves introduced by Bestivina and Handel to show that the eigenvalues $\Lambda(f)$ are non-increasing with respect to the lexicographical order. Since the possible values of $\Lambda(f)$ are finite, the algorithm eventually stops and we obtain an endperiodic relative train track map.

To prove Theorem \ref{thm_1.2}, we compare $\log \lambda(F)$ with the spectral radii of the transition matrices of compact subgraphs of $G$. On the other hand, these spectral radii are also related to the exponential growth rate of lengths of iterates of loops in the subgraph, which represents conjugacy classes in the fundamental group $\pi_1(G)$.

To prove Theorem \ref{thm_1.4}, we define the {\it fundamental group entropy} $h(\pi_1(G),f,\{\ell_i\})$ of a generalized endperiodic map $f \colon G \to G$ with respect to a sequence $\{\ell_i\}$ of finite metrics on $G$. A key observation is that for any two generalized endperiodic map $f_i \colon G_i \to G_i,i=1,2$, there exist sequences of finite metrics $\{\ell_{i,j}\}_{j\ge1}$ on $G_i$ such that $h(\pi_1(G_1),f_1,\{\ell_{1,j}\}) = h(\pi_1(G_2),f_2,\{\ell_{2,j}\})$; see Lemma \ref{lem_6.1}. Then we show that for generalized endperiodic relative train track map $f \colon G \to G$, we have $\log \lambda(f) = h(\pi_1(G),f,\{\ell_i\})$.

\subsection{Organization of the paper} The paper is organized as follows. In Section \ref{sec_prelim}, we give definitions and basic properties of generalized endperiodic maps and realtive train track maps. In Section \ref{sec_moves} we review the ``moves'' in the Bestvina-Handel algorithm, and in Section \ref{sec_thm1.1} we prove Theorem \ref{thm_traintrack}. In Sections \ref{sec_thm1.2} and \ref{sec_pf_thm1.4}, we prove Theorems \ref{thm_1.2} and \ref{thm_1.4} respectively. We record some further questions in Section \ref{sec_questions}.

\section{Definitions} \label{sec_prelim}
In this section, we define endperiodic maps and generalized endperiodic maps in Section \ref{sec_def_end}. Adapting the work of Bestvina-Handel \cite{BH92} to the infinite type setting, we define endperiodic relative train track maps in Section \ref{sec_def_rtt}. In Section \ref{sec_def_cf}, we define a compatible filtration associated to a generalized endperiodic map. Finally in Section \ref{sec_pi1_entropy}, we define the fundamental group entropy of a generalized endperiodic map with respect to (a sequence of) metrics on the graph, which will be used to prove Theorem \ref{thm_1.4}.

\subsection{Endperiodic maps and generalized endperiodic maps} \label{sec_def_end}
An {\it infinite graph} $G$ is a connected cell complex of dimension 1 with countably infinitely many vertices and edges. In this paper we assume that $G$ is {\em locally finite}, i.e., each $0$ cell can only be associated with finitely many $1$ cells. An {\em end} of a graph $G$ is an equivalence class of a nested sequence of connected sets $\{V_n\}$ such that $G \supset V_1 \ldots \supset V_n \supset \overline{V}_{n+1} \supset V_{n+1} \supset \ldots$ and $\cap_{n=1}^\infty V_n = \emptyset$. Two such sequences of sets $\{V_n\}$ and $\{U_m\}$ are {\em equivalent} if for every $n \ge 1$, there exists an $m$ such that $V_n \supset U_m$ and for every $m \ge 1$, there exists an $n$ such that $U_m \supset V_n$. A totally disconnected topology on the space of ends can be defined by setting the open sets as the ends of connected components of $\overline{G\backslash V}$, where $V$ goes through all finite subgraphs.

A {\em graph map} $f \colon G \to G$ is a cellular map which is also a homotopy equivalence. Given a graph map $f : G \to G$, an end $\mathcal{E}$ of $G$ is called {\em attracting} if there is a neighborhood $U = U_{\mathcal{E}}$ of $\mathcal{E}$ such that $f^n(U) \subsetneq U$ for some $n\geq 1$, and $\bigcap_{j=1}^\infty f^{jn}(U)=\emptyset$. An end of $G$ is called {\em repelling} if there is a neighborhood $U = U_{\mathcal{E}}$ of $\mathcal{E}$ such that $f^{-1}(U)\subsetneq U$ for some $n\geq 1$, and $\bigcap_{j=1}^\infty f^{-jn}(U)=\emptyset$. The above neighborhoods $U_{\mathcal{E}}$ are called {\it $f$-neighborhoods} of an end $\mathcal{E}$.

\begin{definition}\label{def_endper}
A graph map $f \colon  G \to G$ is called {\em generalized endperiodic} if $G$ has finitely many ends, and each end is either attracting or repelling. It is called {\em endperiodic} if it is generalized endperiodic and its restriction to the $f$-neighborhoods of the ends are homeomorphisms that send edges to edges. When $f$ is endperiodic, those neighborhoods are called {\em periodic neighborhoods} around the contracting or repelling ends.
\end{definition}

\begin{remark}
A generalized endperiodic graph map is a proper homotopy equivalence, so is an endperiodic graph map.
\end{remark}

\begin{definition}\label{def:cbhe}
A {\em combinatorially bounded homotopy equivalence} is a cellular map $\psi \colon G \to G'$ and its homotopic inverse $\phi$, such that there exists a constant $L>0$, such that every edge $e$ in $G$ is sent to a path of no more than $L$ edges by $\psi$, and every edge $e'$ in $G'$ is sent to a path of no more than $L$ edges by $\phi$.
\end{definition}

\begin{definition} \label{def_tranmatrix}
   Let $f \colon G\rightarrow G$ be a graph map, $V$ a finite subgraph of $G$ consisting of edges $e_1, \dots, e_n$.
   \begin{enumerate}
       \item The {\em transition matrix} $M_{f, V}$ of $f$ on $V$ is an $n\times n$ matrix where the $(i, j)$-th entry equals the number of times the edge path $f(e_j)$ crosses $e_i$ in either directions.
       \item $V$ is called an {\em irreducible} subgraph if the transition matrix is irreducible, i.e., if for any $i, j \in \{1,\ldots,n \}$, there is some $N>0$ such that the $(i, j)$-th entry of $M_{f, V}^N$ is positive.
       \item We say that $V$ is {\em exponentially growing} if the spectral radius $\lambda$ of $M_{f, V}$ is strictly greater than $1$.
   \end{enumerate}
\end{definition}

We give examples of generalized endperiodic maps and endperiodic maps.
\begin{example}
\begin{enumerate}
\item Let $G_0$ be a connected finite graph, $G$ an infinite cyclic cover of $G_0$, such that $h$ is a generator of the deck group. Let $g$ be a compactly supported self map on $G$, then $g\circ h$ is endperiodic with one attracting and one repelling end.
\item Let $G_0$ be a connected finite graph, $f\colon G_0\rightarrow G_0$ a cellular map which is a homotopy equivalence and can be lifted to an infinite cyclic cover $G$ as $f'$. Let $g$ be a compactly supported self map on $G$ and $h$ be a generator of the deck group. Then for $n \gg 0$, the map $g\circ h^n \circ f'$ is generalized endperiodic but not endperiodic in general.
\end{enumerate}
\end{example}

\subsection{Relative train track maps} \label{sec_def_rtt}
In this section, we extend the definition of relative train track maps of Bestvina-Handel \cite{BH92} to the infinite setting.

A {\it filtration} of a graph map $f \colon  G \to G$ is an increasing sequence of (not necessarily connected) invariant subgraphs $G_0 \subset \cdots \subset G_m = G$ for some $m \in \mathbb N$. The subgraph $H_i \defeq {\rm cl}(G_i\setminus G_{i-1})$ (where $G_{-1}=\emptyset$) is called the {\it $i^{\rm th}$ stratum}. When $H_i$ is finite, let $M_i$ be the transition matrix for $H_i$.

Following Bestvina-Handel \cite{BH92}, by a {\em turn} in $G$, we mean an unordered pair of oriented edges of $G$ originating at a common vertex. A turn with one edge in $H_i$ and one edge in $G_{i-1}$ is called a {\em mixed turn} in $(G_i,G_{i-1})$. A turn is {\it non-degenerate} if it is not an oriented edge followed by its inverse, and is {\it degenerate} if otherwise. A map $f \colon G \to G$ induces a self-map $Df$ on the set of oriented edges of $G$ by sending an oriented edge to the first oriented edge in its $f$-image; this induces a map $Tf$ on the set of turns in $G$.

A turn is {\it illegal} with respect to $f \colon G\to G$ if its image under some iterate of $Tf$ is degenerate; a turn is {\it legal} otherwise. A {\it path} in $G$ is a sequence of consecutive edges $e_0e_1\dots e_d$. We call a path {\em legal} if all the turns $\{e_i, e_{i+1}\}$ are legal.

\begin{definition}[{\cite{BH92}}]
A cellular homotopy equivalence $f \colon G\to G$ is called a {\em train track map} if the $f$-image of every edge in $G$ is a path without illegal turns.
\end{definition}

We extend the definition of relative train tracks in \cite{BH92} to the infinite graph case as follows.

\begin{definition} \label{def_reltt}
A cellular homotopy equivalence $f \colon G\to G$ of an infinite graph $G$ is called a {\em relative train track map} if it fixes a filtration $G_0\subseteq G_1 \subseteq \dots \subseteq G_m=G$, such that
\begin{enumerate}[1.]
    \item The only possible infinite strata are $H_0=G_0$ and $H_m=\overline{G_m\backslash G_{m-1}}$. Given any compact subgraph $G'$, any edge $e$ in an infinite stratum, there is some $n$ such that $f^{n'}(e)\cap G'=\emptyset$ for all $n'>n$ (in this case we call $e$ an {\em escaping edge}) or $f^{-n'}(e)\cap G'=\emptyset$ for all $n'>n$ (in this case we call $e$ a {\em backward escaping edge}).
    \item If $H_r$ is a finite strata which is exponentially growing (see Definition \ref{def_tranmatrix}), then it is irreducible, and furthermore:
\begin{enumerate}[(1)]
\item $Df$ maps the set of oriented edges in $H_r$ to itself; in particular, all mixed turns in $(G_r,G_{r-1})$ are legal.
\item If $\alpha \subset G_{r-1}$ is a nontrivial path with endpoints in $H_r \cap G_{r-1}$, then $[f(\alpha)]$ is a nontrivial path with endpoints in $H_r \cap G_{r-1}$.
\item For each legal path $\beta \subset H_r$, $f(\beta)$ is a path that does not contain any illegal turns in $H_r$.
\end{enumerate}
\end{enumerate}
\end{definition}

\subsection{Compatible filtrations} \label{sec_def_cf}
Given a generalized endperiodic map $f \colon  G\to G$, we construct a filtration of $G$, called the {\em compatible filtration}, whose finite components are irreducible.

\begin{definition}\label{ini_filtration}
Let $f \colon G\to G$ be a generalized endperiodic map. Let $G_0$ be the subgraph of $G$ consisting of escaping edges, $G'$ be the subgraph consisting of all non backward escaping edges. Then all the $f$-neighborhoods of ends are contained in either $G_0$ or $G\backslash G'$, hence $\overline{G'\backslash G_0}$ is a finite graph. Now build a filtration $G_0\subsetneq G_1\subsetneq \dots \subsetneq G_{m-1}=G'\subset G$, such that each component that is exponentially growing (see Definition \ref{def_tranmatrix}) is irreducible. This filtration is called the {\em compatible filtration}.
\end{definition}

Given a generalized endperiodic map $f$ and a compatible filtration $G_0\subseteq \dots \subseteq G_m=G$, we let $\lambda_i$ be the spectral radius of the $M_i$ as in Definition \ref{def_reltt} when $H_i$ is finite. If $H_i$ is infinite, by construction all edges in $H_i$ are escaping or backward escaping, so we set the corresponding $\lambda_i$ to be $0$. Let $\Lambda(f)$ be the set of $\lambda_i$ which are strictly greater than $1$ ordered from large to small, and $\lambda(f)$ be the maximal element in $\{\lambda_i\}$.

\subsection{Fundamental group entropy} \label{sec_pi1_entropy}
Given a generalized endperiodic map $f \colon G \to G$ and a sequence of finite metrics $\{\ell_i\}$ on $G$, in this section, we define the fundamental group entropy $h(\pi_1(G),f,\{\ell_i\})$ and study its relationship with the topological entropy $h_{\rm top}(f)$ of $f$.

\begin{definition}
Given a generalized endperiodic map $f \colon G \to G$ and a metric $\ell$ on $G$ such that the sum of all edge lengths is finite,
we define the {\it fundamental group entropy $h(\pi_1(G),\ell)$ with respect to $\ell$} by
$$h(\pi_1(G),\ell) \defeq \sup_\mathfrak{c} \limsup_{n \to \infty} \frac{\log \ell(f^n(\mathfrak{c}))}{n}$$
where the supremum is taken over all the {\color{black} conjugacy classes} $\mathfrak{c}$ in $\pi_1(G)$, and $\ell(\mathfrak{c})$ for a conjugacy class $\mathfrak{c}$ is the length of the shortest loop representing $\mathfrak{c}$ as measured by $\ell$.
\end{definition}

We say that $\ell\lesssim \ell'$ if there is some $C>0$ such that $\ell(e) \leq C\ell'(e)$ for any edge $e$ in $G$.

\begin{definition}
Given a generalized endperiodic map $f \colon G \to G$ and a sequence $\{\ell_i\}$ of metrics on $G$, such that $\ell_i\lesssim \ell_{i+1}$ for every $i$, the {\it fundamental group entropy $h(\pi_1(G),f,\{\ell_i\})$ with respect to $\{\ell_i\}$} is defined as
$$h(\pi_1(G),f, \{\ell_i\}) \defeq \lim_{i \to\infty} h(\pi_1(G),\ell_i) \in \mathbb{R}\cup \{\infty\}.$$
\end{definition}
\begin{remark} \label{rmk_2.11}
The above definition depends on the choice of the sequence $\{\ell_i\}$ of metrics on $G$. However, if we have two sequences $\{\ell_i\}$ and $\{\ell_i'\}$ such that for any sufficiently large $i$, there are $j, j'$ such that $\ell'_j\lesssim \ell_i\lesssim \ell'_{j'}$, then $h(\pi_1(G), f, \{\ell_i\})=h(\pi_1(G), f, \{\ell'_i\})$. We call such two sequences {\it interleaving}.
\end{remark}

\begin{definition}\label{l_idef}
Let $f\colon G \to G$ be a generalized endperiodic map. Let $\{G_i\}$ be a compact exhaustion of $G$, such that the complement consists of contracting or repelling neighborhoods of the ends. For each $i$, let $\ell_i$ be defined as assigning length $1$ to all edges in $G_i$ and $0$ to all other edges. 
\end{definition}

\begin{remark}\label{rmk_topentropy}
{\color{black} The graph $G$ is not compact. However, since the endperiodic map $f \colon G \to G$ is proper, it induces a continuous map from $\overline{G}$, which is the one-point compactification of $G$, to itself. Denote by $h_{\rm top}(f)$ the topological entropy of this induced map on the compact space $\overline{G}$.}
\end{remark}

\begin{lemma} \label{lem_2.130}
Let $f \colon G \to G$ be a generalized endperiodic map, $\{\ell_i\}$ be defined as in Definition \ref{l_idef}. Then we have $h(\pi_1(G),f,\{\ell_i\}) \le h_{\rm top}(f)$.
\end{lemma}
\begin{proof}
{\color{black}We only need to show that for any $\ell_i$, 
we have $h(\pi_1(G), \ell_i)\leq h_{\rm top}(f)$. Pick a metric $\ell$ on $\overline{G}$ such that edges of $G_i$ all have length $1$. Let $\lambda=e^{h(\pi_1(G), \ell_i)}$. Then for any $\epsilon>0$, there is some conjugacy class $\mathfrak{c}$ represented by a loop $\gamma$ such that
\[\ell(f^n(\gamma))\geq \ell_i(f^n(\gamma))\geq \ell_i(f^n(\mathfrak{c}))\geq O((\lambda-\epsilon)^n)\]
as $n\to\infty$. Now pick some sufficiently small $\epsilon'$, we get $(n, \epsilon')-$separated sets on $\gamma$ under metric $\ell$ of size $O((\lambda-\epsilon)^n)$. This shows that $h_{top}(f)\geq \log(\lambda-\epsilon)$. We obtain the conclusion by letting $\epsilon\to 0$.
}


\end{proof}

\section{Moves} \label{sec_moves}
In this section, we recall the moves that were introduced by Bestvina-Handel \cite{BH92} to be applied to a graph map in order to make it (relatively) efficient. We summarize the moves below and refer the reader to \cite[Section 1]{BH92} for details.
\begin{enumerate}
\item Pulling tight: A homotopy equivalence $f \colon G \to G$ is {\it tight} if for each edge $e$ of $G$, either $f$ is locally injective on the interior of $e$ or $f(e)$ is a vertex. A homotopy equivalence $f \colon G \to G$ can be {\it pulled tight} by a homotopy rel the set of vertices of $G$.
\item Collapsing an invariant forest: If the graph map $g \colon G \to G$ has an invariant forest $G_0$, i.e., a subgraph each of whose components is contractible, then define $G_1 \defeq G/G_0$ and define $f_1 \colon G_1 \to G_1$ by $f_1 \defeq \pi f \pi^{-1}$ where $\pi \colon G \to G_1$ is the quotient map.
\item Subdivision: Suppose $w$ is not a vertex of $G$ but $f(w)$ is a vertex of $G$. Then we {\it subdivide} the edge containing $w$ by adding a vertex at $w$.
\item Valence-one homotopy: Suppose $f \colon G \to G$ is a homotopy equivalence and $v$ is a valence-one vertex of $G$ with incident edge $e$. The goal of a valence-one homotopy is to remove $e$ in the image of $f$. To this end,
let $G_1$ be the subgraph of $G$ obtained by removing $v$ and the interior $e$. Define $f_1 \colon G_1 \to G_1$ by $f_1 \defeq \pi_1f|_{G_1}$ where $\pi_1 \colon G \to G_1$ is the projection map. Then $f_1$ is said to be obtained from $f$ by a {\it valence-one homotopy}.
\item Valence-two homotopy: Suppose $v$ is the initial vertex of an edge $e_m$ and the terminal vertex of an edge $e_{m-1}$. A {\it valence-two homotopy} is a homotopy which collapses $e_m$ to a point and stretches $e_{m-1}$ across $e_m$, followed by a tightening, an inverse subdivision to remove $v$ and collapsing a maximal pretrivial forest.
\item Folding: The goal of folding is to eliminate a pair of edges originating at a common vertex of $G$ which have the same images under $f$. Let $G_1$ be the graph obtained from $G$ by identifying these two edges such that $f \colon G\to G$ descends to a well-defined map $f_1 \colon G_1 \to G_1$. Then $f_1$ is said to be obtained from $f$ by an {\it elementary fold}. In general, we can fold a subset of an edge with a subset of another edge.
\end{enumerate}

\section{Proof of Theorem \ref{thm_traintrack}} \label{sec_thm1.1}
In this section, we prove Theorem \ref{thm_traintrack}. The proof consists of two steps. We first show in Section \ref{sec_safe_moves} that we can apply Bestvina-Handel moves defined in the previous section to make the vector of stretch factors $\Lambda(f)$ minimal. In Section \ref{sec_4.2}, we show that in the proper homotopy class of maps with $\Lambda(f)$ minimal, there is a relatvie train track map. Finally, in Section \ref{ex2}, we give an example to demonstrate the algorithm given in the proof of Theorem \ref{thm_traintrack}.

Let $f \colon G \to G$ be a generalized endperiodic map. 
We consider a compatible filtration $\emptyset = G_{-1} \subset G_0 \subset G_1 \subset \cdots \subset G_{m-1} \subset G_m = G$ of $(G,f)$ as in Definition \ref{ini_filtration}.

\subsection{Safe moves} \label{sec_safe_moves}
Recall the moves in Section \ref{sec_moves}. A move is {\it safe} if it does not increase $\Lambda(f)$ and is {\it beneficial} if it strictly decreases $\Lambda(f)$. We first analyze the moves to determine if they are safe. Our analysis works almost the same way as Bestivina-Handel \cite[Section 5]{BH92}, except we need to take care of the repelling ends.

The first lemma states that subdivision is safe and that valence-one homotopy and folding are safe and sometimes beneficial; see also \cite[Lemmas 5.1-5.3]{BH92}.

\begin{lemma}\label{lemsafemv}
	We have the following:
	\begin{enumerate}
		\item If $f'\colon G'\to G'$ is obtained from $f \colon G\to G$ by subdivision, then $\Lambda(f') = \Lambda(f)$.
		\item If $f'\colon G'\to G'$ is obtained from $f \colon G\to G$ by a valence-one homotopy at a vertex $v$, then $\Lambda(f') \le \Lambda(f)$. Moreover, if the stratum $H_r$ containing the edge incident to $v$ is exponentially growing, then $\Lambda(f') < \Lambda(f)$.
		\item Suppose $f'\colon G'\to G'$ is obtained from $f \colon G\to G$ by folding the entire edge $E_i \subset H_i$ with the entire edge $E_j \subset H_j$ and $f'':G''\to G''$ is obtained from $f' \colon G'\to G'$ by tightening and collapsing a pretrivial forest. If $H_i = H_j$ is exponentially growing and if the nontrivial tightening happens in the stratum of $G'$ determined by $H_i$, then $\Lambda(f'') < \Lambda(f)$. Otherwise, we have $\Lambda(f'') = \Lambda(f)$.
	\end{enumerate}
\end{lemma}
\begin{proof}
	For (1), let $e$ in $H_i'$ be an edge that is subdivided into two edges. By the same argument as in the proof of \cite[Lemma 5.1]{BH92}, we have $\lambda_i' = \lambda_i$. Furthermore, if $e$ is in a fundamental domain of an repelling end, we subdivide all the preimages of $e$ in the fundamental domains of the repelling end.

	For (2), recall that a valence-one homotopy is to remove the edge at a valence one vertex. It consists of a homotopy followed by a tightening and a collapsing of a pretrivial forest. If any edges in a fundamental domain of a repelling end is collapsed, we need to collapse all its preimages in the repelling end. For $l \neq r$, if $M_l$ is irreducible, then $M_l$ remain unchanged. If $M_l$ is a zero matrix, then the edges that are not collapsed in $H_l$ determines a new zero matrix. For $H_r$, the edges that are not collapsed in $H_r$ may determine one or more stata. If the original $H_r$ is exponentially growing and one of the new stratum is irreducible, then the transition matrix of this new stratum is a square submatrix of $M_r$, and the Perron-Frobenius eigenvalue is strictly smaller than the Perron-Frobenius eigenvalue of $H_r$ by the proof of \cite[Lemma 1.11]{BH92}.

	We now prove (3). If $l \neq i,j$, then similar to (2), if $M_l$ is irreducible, then $M_l$ remain unchanged. If $M_l$ is a zero matrix, then the edges that are not collapsed in $H_l$ determines a new zero matrix. If $i < j$, then $H_j$ is a zero matrix as the image $f(E_j)$ is contained in $G_i$ with $i<j$. The edges that are not folded or collapsed in $H_j$ determine a zero transition matrix. The matrix $H_i$ remains unchanged. Therefore $\Lambda(f'') = \Lambda(f)$ remains the same. If $i = j$ and if $H_i$ is not exponentially growing, then $\lambda_i$ remains the same. If $i = j$ and if $H_i$ is exponentially growing, then $\lambda_i$ remains the same if there is no tightening. Otherwise $\lambda_i$ will become strictly smaller due to the tightening operation, as the edges that are not tightened give a square submatrix of the original $M_i$ and the Perron-Frobenius eigenvalue strictly decreases again by the proof of \cite[Lemma 1.11]{BH92}. If $i=j=m$ and the two edges are in a fundamental domain of a repelling end, then we fold all their preimages in the repelling end. If $i=j=m$ and one of the edge is in a fundamental domain, then we collapse all the preimages of that edge in the repelling end.
\end{proof}

The next lemma analyzes the valence-two homotopy, which can be dangerous if $H_i = H_j$ is exponentially growing. See also \cite[Lemma 5.4]{BH92}.
\begin{lemma}
Suppose that $E_i \subset H_i$ and $E_j \subset H_j$ are two edges incident to a valence two vertex $v$ with $i \le j$. If $i=j$ and $H_i$ is exponentially growing, then up to relabeling, the eigenvector coefficient of $E_i$ is greater than or equal to that of $E_j$. Suppose $f' \colon G' \to G'$ is obtained from $f \colon G \to G$ by a homotopy across $E_i$. Then
	\begin{enumerate}
		\item if $H_i$ is not exponentially growing, then $\Lambda(f') = \Lambda(f)$;
		\item if $i<j$ and $H_i$ is exponentially growing, then $\Lambda(f') < \Lambda(f)$;
		\item if $i=j$ and $H_i$ is exponentially growing, then $\lambda_i$ is replaced by some number $\lambda'_{ij}$'s, each of which satisfies $\lambda'_{ij} \le \lambda_i$.
	\end{enumerate}
\end{lemma}

\begin{proof}
The valence-two homotopy move consists of a homotopy followed by a tightening, an inverse subdivision followed by a tightening and a collapse of a pretrivial forest. We first note that $\lambda_l$ will remain the same for $l < i$ as there is no effect on $G_l$. If $l>i$, $l \neq j$ and $H_l$ is exponentially growing, then $\lambda_l$ remains the same as well, as the move erases $E_i$ and $\overline{E}_i$ in the images $f(e)$ of edges $e$ in $H_l$ and does not change the transition matrix.

Suppose $i<j$ and $H_j$ is irreducible. Then $\lambda_j$ stays unchanged as the move erases $E_i$ and $\overline{E}_i$ in the images $f(e)$ of edges $e$ in $H_j$ with $e \neq E_j$ followed by a tightening, and there is no collapsing happening. The effect of the move on $H_i$ is the same as the valence-one homotopy. Therefore $\lambda_i$ is unchanged if $H_i$ is not exponentially growing and $\lambda_i$ is replaced by some $\lambda_{ij}$'s with $\lambda_{ij} < \lambda_i$ if $H_i$ is exponentially growing.

Suppose $i=j$ and  if $H_i$ is irreducible. By \cite[Lemma 1.13]{BH92}, if the matrix $M_i(2)$ for the map $f(2) \colon G(2) \to G(2)$ in the definition of valence-two homotopy is irreducible, then $\lambda'_i \le \lambda_i$. If $M_i(2)$ is not irreducible, then $H_i$ will split into strata $H_{ij}'$ in $G'$. The transition matrix $M_{ij}'$ of each stratum $H_{ij}'$ is a square submatrix of $M_i(2)$ and therefore $\lambda_{ij}' \le \lambda_i$.

If $i=m$ and $E_i$ is in a fundamental domain of a repelling end, then we collapse all the preimages of $E_i$ in the repelling end, as in the case of valence-one homotopy.
\end{proof}

Since the valence-two homotopy may be dangerous, we cannot require the graph $G'$ to have no valence-two vertices. As an alternative, we use the following definition and Lemma \ref{lem_bdd}.

\begin{definition} [\cite{BH92}]\label{def_bdd}
Let $L$ be the number of edges in $(G_{m-1}\backslash G_0) \cup \overline{H_m\backslash f^{-1}(H_m)}$. We say that $f:G\to G$ is {\it bounded} if there are at most $L$ exponentially growing strata $H_r$ and if for each such $H_r$, $\lambda_r$ is the Perron-Frobenius eigenvalue for some matrix with no more than $L$ rows and columns.
\end{definition}

\begin{remark} \label{rmk_lambda_min}
Since there are only finitely many possible values for $\lambda_i$'s, $\Lambda_{\min}$ exists among all the $\Lambda(f)$ with $f$ a bounded endperiodic map.
\end{remark}

\begin{lemma} \label{lem_bdd}
Suppose $f' \colon G' \to G'$ is obtained from a bounded endperiodic map $f:G \to G$ by a sequence of safe moves such that $\Lambda(f') < \Lambda(f)$. Then there exists a bounded endperiodic map $f'' \colon G'' \to G''$ such that $\Lambda(f'') < \Lambda(f)$.
\end{lemma}
\begin{proof}
The proof follows that of \cite[Lemma 5.5]{BH92}. The idea is to modify $f'$ so that it becomes bounded and $\Lambda(f')$ will not increase too much to be bigger than $\Lambda(f)$. We first note that $G'$ contains no valence-one vertices and if $v$ is a valence-two vertex in $G'$, both edges incident to $v$ lie in the same stratum in $G'$. If $v$ and both the edges are contained in a new stratum of $G'$, i.e., a stratum of $G'$ which is not a stratum of $G$, then we apply a valence-two homotopy move at $v$. If the edge $E_i$ incident to $v$ is in a fundamental domain of a repelling end, then we collapse all the preimages of $E_i$ in the repelling end. We obtain $f'' : G'' \to G''$ with $\Lambda(f'') < \Lambda(f)$. The filtration for $f''$ has at most $L$ strata and each new stratum has at most $L$ edges. Therefore $f''$ is bounded.
\end{proof}

Therefore, we obtain the following conclusion.

\begin{proposition} \label{prop_1}
If $f \colon G\to G$ is a generalized endperiodic (resp. endperiodic) map, then $f$ is homotopic to a generalized endperiodic (resp. endperiodic) map with $\Lambda = \Lambda_{\min}$.
\end{proposition}
\begin{proof}
If $f \colon G\to G$ is a generalized endperiodic map,
the proposition follows from Lemma \ref{lem_bdd} and Remark \ref{rmk_lambda_min}. If $f \colon G\to G$ is endperiodic, since all the moves are performed in a compact subgraph disjoint from the ends, the final map with $\Lambda(f)$ minimal is endperiodic as well.
\end{proof}

\subsection{Train track maps and relative train track maps}\label{sec_4.2}

\begin{proposition}\label{prop_2}
If $f \colon G\to G$ is a generalized endperiodic (resp. endperiodic) map with $\Lambda = \Lambda_{\min}$, then $f$ is homotopic to a generalized endperiodic (resp. endperiodic) relative train track map.
\end{proposition}
\begin{proof}
Since $f:G\to G$ is a generalized endperiodic map with $\Lambda = \Lambda_{\min}$, by \cite[Lemma 5.9]{BH92}, if $f$ satisfies Definition \ref{def_reltt} (1), then it satisfies Definition \ref{def_reltt} (3). Therefore we modify $f$ and $G$ so that Definition \ref{def_reltt} (1) and (2) hold.

Starting from $H_1$ and working inductively, if $H_i$ does not satisfy Definition \ref{def_reltt} (1), then apply the {\em core subdivision} to $H_i$ as in \cite[Lemma 5.13]{BH92} to obtain an endperiodic map $f(1) : G(1) \to G(1)$ with $\Lambda(f(1)) = \Lambda(f) = \Lambda_{\min}$ such that $H_i(1)$ satisfies Definition \ref{def_reltt} (1). The ``core subdivision'' procedure is done as follows: for every edge in $H_i$, let the {\em core} be all the interior points in $H_i$ whose forward orbit does not escape to $G_{i-1}$. Consider the smallest interval on this edge that contains the core, if an end point of this interval is not a vertex, make it a vertex (and hence subdivided the edge), and set $G_{i-1}$ to also include the segment not in the core.

Now if it fails to satisfy Definition \ref{def_reltt} (2), apply {\em collapsing inessential connecting paths} to $H_i$ as in \cite[Lemma 5.14]{BH92}. The ``collapsing inessential connecting paths'' procedure is done as follows: If some $[f(\alpha)]$ is trivial, carry out a sequence of folds to eliminate it, and collapsing the pretrivial forest. This procedure will result in a decrease of the number of points in $H_r\cap G_{r-1}$, hence such $\alpha$ can be eliminated after finitely many such operations.
Note that both $H_0$ and $H_m$ are non exponentially growing.

If $f \colon G\to G$ is endperiodic, since the core subdivision and collapsing inessential connecting paths are performed in a compact subgraph disjoint from the ends, the final map is endperiodic as well.

The proof is completed.
\end{proof}

\begin{proof}[Proof of Theorem \ref{thm_traintrack}]
If $f \colon G\to G$ is generalized endperiodic, the first sentence of the theorem follows from Propositions \ref{prop_1} and \ref{prop_2}, and the fact that all the operations in the proof of those propositions are carried out in a compact subgraph so that the conjugation to the relative train track is via combinatorially bounded homotopy equivalences. Moreover, if $f \colon G\to G$ is furthermore endperiodic, we note that all the operations in the proofs of Propositions \ref{prop_1} and \ref{prop_2} are carried out in a compact subgraph so that they would not destroy sufficiently small periodic ends. Therefore the conclusion for endperiodic maps follows.

Now to prove the second sentence of the theorem, we only need to show that from an endperiodic relative train track map with a single finite strata which is exponentially growing, one can obtain a train track map via finitely many moves as defined in Section \ref{sec_moves}. By collapsing pretrivial forests we can assume that no edge would be sent to a single point, hence the only obstruction to being train track is that $f^n$ might send an edge to a path with backtrackings. Let $f \colon G\rightarrow G$ be such a map, $G'$ be $G_1$ deleting its intersection with the union of periodic neighborhoods of repelling ends. Hence $f(G)\subseteq G$, and edges in $G'$ are either in the exponentially growing strata or in $H_0$, and by the assumption of $f$ being endperiodic, all but finitely many of those in $H_0$ are in periodic ends of attracting ends. By Definition \ref{def_reltt}, we know that if the iterated forward image of some edge of $G'$ has backtracking, the backtracking itself must appear in $H_0$. By definition of end periodicity, the illegal turn that get sent to the backtracking cannot be in the periodic neighborhoods of the attracting ends. We define the {\em height} of an illegal turn in $G'$ as the number of iterates of $f$ needed to get into the periodic neighborhoods, then each backtracking is associated with an illegal turn with finite, positive height. Carry out the folding, tightening and collapsing pretrivial forest process as in Lemma \ref{lemsafemv} Part (3) to the illegal turns associated with backtracking with the lowest height, then the lowest possible height of such turns would increase. After finitely many such operations one can remove backtracking in the forward iterated images of edges entirely in $G'$. Now subdivide the edges in $G\backslash G'$ such that $f$ sends each edge there to a single edge, and we get an endperiodic train track map.
\end{proof}

\subsection{An example}\label{ex2}
To illustrate the algorithm described in the proof of Theorem \ref{thm_traintrack}, we consider the following example, which is based on \cite[Example 4.13]{CCF21}.

Let $G$ be an infinite graph with two ends illustrated as below in Figure \ref{fig:example0}:

\begin{figure}[H]
    \centering
    \begin{tikzpicture}[scale=1.5]
    \draw[-](0.5, 0)--(5.5, 0);
    \draw[-](0.5, 1)--(5.5, 1);
    \draw[-](1, 1)--(1, 0);
    \draw[-](2, 1)--(2, 0);
    \draw[-](3, 1)--(3, 0);
    \draw[-](4, 1)--(4, 0);
    \draw[-](5, 1)--(5, 0);
    \node at (3.2, 0.5) {$a_0$};
    \node at (2.2, 0.5) {$a_{-1}$};
    \node at (1.2, 0.5) {$a_{-2}$};
    \node at (4.2, 0.5) {$a_1$};
    \node at (5.2, 0.5) {$a_2$};

    \node at (0.5, 1.2) {$b_{-3}$};
    \node at (1.5, 1.2) {$b_{-2}$};
    \node at (2.5, 1.2) {$b_{-1}$};
    \node at (3.5, 1.2) {$b_0$};
    \node at (4.5, 1.2) {$b_1$};
    \node at (5.5, 1.2){$b_2$};

    \node at (0.5, -.2) {$c_{-3}$};
    \node at (1.5, -.2) {$c_{-2}$};
    \node at (2.5, -.2) {$c_{-1}$};
    \node at (3.5, -.2) {$c_0$};
    \node at (4.5, -.2) {$c_1$};
    \node at (5.5, -.2){$c_2$};
    \draw[dashed](0, 0)--(0.5, 0);
    \draw[dashed](0, 1)--(0.5, 1);\draw[dashed](5.5, 0)--(6, 0);\draw[dashed](5.5, 1)--(6, 1);
    \end{tikzpicture}
    \caption{Graph for Example in Section \ref{ex2}}
    \label{fig:example0}
\end{figure}
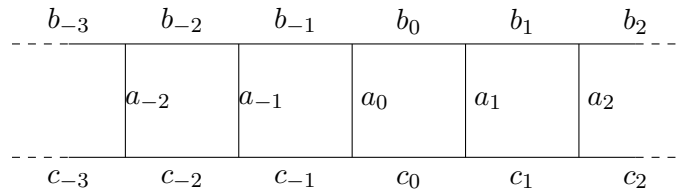

Let $g$ be a shift to the right, sending $a_i$, $b_i$, $c_i$ to $a_{i+1}$, $b_{i+1}$ and $c_{i+1}$ respectively. Let $\tau$ be a map that sends $a_0$ to the path $b_{-1}a_{-1}c_{-1}c_0a_1b_0a_0c_{-1}a_{-1}b_{-1}b_0a_1c_0$ and fixing all other edges. Consider the map $g\circ \tau : G \to G$. In this case $m=3$ for the filtration. $H_1$ consists of $b_i$, $c_i$ where $i\geq -1$, and $a_i$ where $i\geq 0$; the exponential component $H_2$ consists of $a_{-1}$, and $H_3$ consists of $a_i$ where $i\leq -2$, and $b_i$, $c_i$ where $i\leq -2$. Definition \ref{def_reltt} (1) is not satisfied so we need to do a core subdivision for $a_{-1}$, then both (1) and (3) holds. One can further verify that (2) holds as well. So the only exponential component is $H_1$ which consists of a single interval, and $\lambda_2=2$.

\section{Proof of Theorem \ref{thm_1.2}} \label{sec_thm1.2}
In this section, we give a proof of Theorem \ref{thm_1.2}.
Recall that for a finite subgraph $G'$ of $G$, the term $l_w(P_{G'}(\cdot))$ is defined as follows. Collapsing all edges in its complement $G \setminus G'$, we get a finite graph $G_1$ and a group homomorphism $P_{G'} \colon \pi_1(G)\rightarrow \pi_1(G_1)$. We denote by $l_w$ the word length on the finitely generated free group $\pi_1(G_1)$. Since $G_1$ is a finite graph, $l_w(P_{G'}(\cdot))$ is bounded from both above and below by a constant multiple of the {\it bounded combinatorial length} $l_{G'}(\cdot)$ that we define now.
For each conjugacy class $\mathfrak{c}$ in $\pi_1(G)$, we define $l_{G'}(\mathfrak{c})$ to be the number of edges in $\gamma \cap G'$, where $\gamma$ is a shortest (tight) loop $\gamma$ in $G$ representing $\mathfrak{c}$.

Therefore, to prove Theorem \ref{thm_1.2}, it is equivalent to use $l_{G'}(\cdot)$ rather than $l_w(P_{G'}(\cdot))$. The proof is based on the following lemma.
Recall Definition \ref{def:cbhe}.

\begin{lemma}\label{bdcomb}
Let $f: G\rightarrow H$ be a combinatorially bounded homotopy equivalence. Let $\mathfrak{c}$ be a conjugacy class in $\pi_1(G)$, $G'$ a finite subgraph of $G$. Then there is some $L>0$ such that $l_{f(G')}(f_*(\mathfrak{c}))\leq L \cdot l_{G'}(\mathfrak{c})$.
\end{lemma}

If $G$ is a graph, by a {\em tight loop} we mean a loop that is the shortest among its homotopy class, i.e., a loop with no backtracking. By {\em tightening} of a loop we mean replacing a loop with the shortest loop in its homotopy class.

\begin{proof}[Proof of Lemma \ref{bdcomb}]
Let $\gamma$ be the tight loop representing $\mathfrak{c}$ and $\gamma'$ be the tightening of $f(\gamma)$. Then by Definition \ref{def:cbhe}, the number of edges of $\gamma'$ in $f(G')$ is no more than the number of edges of $f(\gamma)$ in $f(G')$, which is no more than $L$ times the number of edges of $\gamma$ in $G'$.
\end{proof}

Now we prove Theorem \ref{thm_1.2}.

\begin{proof}[Proof of Theorem \ref{thm_1.2}]
Let $F\colon G_t\rightarrow G_t$ be the relative train track representation of $f \colon G \rightarrow G$, $h\colon G\rightarrow G_t$ a homotopy equivalence represented by combinatorially bounded map, such that $h\circ f$ and $F \circ h$ are homotopic, $h'$ the homotopy inverse of $h$.

We prove Statement (1). If $\lambda(f)=0$, by Definition \ref{def_reltt}, then every edge in $G_t$ will be escaping under $F$. This is because if a finite matrix with non negative entries has spectral radius $0$, then it becomes $0$ after taking finite power. Let $G'$ be any compact subgraph of $G$, $\gamma$ any loop, $\mathfrak{c}$ the conjugacy class represented by $\gamma$. Then by Lemma \ref{bdcomb}, we have $$l_{G'}(f_*^n(\mathfrak{c}))\leq L \cdot l_{h'^{-1}(G')}(f_*^n(h_*(\mathfrak{c})))$$ for some $L>0$. We note that there exists $n \in \mathbb{N}$ such that the right hand side equals $0$ when $n' \ge n$. Therefore the conclusion follows.

Now we prove Statement (2). By Lemma \ref{bdcomb}, it suffices to show that:
\begin{equation}\label{eq_last}
\log(\lambda(F))=\sup_{H'\subseteq G_t\text{ compact},~ \mathfrak{c} \text{ conjugacy class of }\pi_1(G_t)}\limsup_{n\rightarrow\infty}\frac{\log(l_{H'}(F_*^n(\mathfrak{c})))}{n}.
\end{equation}

\begin{enumerate}
\item To show $\geq$, note that by definition, $\lambda(F)$ is the maximum among the spectral radii of the matrices $M_{H'}=[m_{i, j}]$, where
\[m_{i, j} \defeq \text{ number of times }f(j\text{-th edge in }H')\text{ passes through }i\text{-th edge in }H'\]
and $H'$ goes through all compact subgraphs of $G'$. Therefore the conclusion follows.
\item Suppose $\lambda(F)=\lambda_i$. Set $H'=H_i$ and $\mathfrak{c}$ to be represented by a non-trivial loop $\gamma$ in $G_i$ which passes through at least one edge in $H_i$. By Definition \ref{def_reltt}, $$l_{H'}(F_*^n(\mathfrak{c}))=[1 \dots 1](M_i)^n x_\gamma$$ where $x_\gamma$ is an $n \times 1$ vector whose $j$-th entry is the number of times $\gamma$ passes through the $j$-th edge in $H'=H_i$ in either directions. Hence the exponential growth rate of $l_{H'}(F_*^n(\mathfrak{c}))$ is the spectral radius of $M_i$, which is $\lambda_i=\lambda(f)$. Therefore $\log(\lambda(F))$ is less than the supremum on the right hand side of \eqref{eq_last}. Moreover, the supremum is realized by $H'$ and $\mathfrak{c}$.
\end{enumerate}
This completes the proof of the theorem.
\end{proof}

\section{Proof of Theorem \ref{thm_1.4}} \label{sec_pf_thm1.4}
Recall from Remark \ref{rmk_topentropy} that the topological entropy $h_{\rm top}(f)$ is well-defined for generalized endperiodic maps.
\begin{lemma} \label{lem_2.13}
Let $f \colon G \to G$ be a generalized endperiodic relative train track map with $\lambda(f)\ge 1$ and $\{\ell_i\}$ be defined as in Definition \ref{l_idef}. Then we have $\log\lambda(f) = h(\pi_1(G),f,\{\ell_i\}) = h_{\rm top}(f)$.
\end{lemma}
\begin{proof}
We first show that $\log \lambda(f) = h(\pi_1(G),f,\{\ell_i\})$. By the proof of Theorem \ref{thm_1.2},
\[\log\lambda(f)=\limsup_{n\rightarrow\infty}\frac{\log(l_{G'_\diamond}(f_*^n(\mathfrak{c}_\diamond)))}{n}\]
where $G_\diamond'$ and $\mathfrak{c}_\diamond$ are the ones realizing the supremum. Let $I \in \mathbb{N}$ be such that the subgraph $G_I$ as in Definition \ref{l_idef} contains $G'_\diamond$. Then for any $i \ge I$, by definition of $\{\ell_i\}$ (as in Definition \ref{l_idef}), we have
$$\limsup_{n\rightarrow\infty}\frac{\log(l_{G'_\diamond}(f_*^n(\mathfrak{c}_\diamond)))}{n} = \limsup_{n\rightarrow\infty}\frac{\log(\ell_i(f^n(\mathfrak{c}_\diamond)))}{n}.$$
In particular, the supremum in the definition of $h(\pi_1(G),f,\{\ell_i\})$ is realized by $\mathfrak{c}_\diamond$. Hence the conclusion follows.

Now we show that $h_{\rm top}(f) = h(\pi_1(G),f,\{\ell_i\})$. Again, when consider $h_{\rm top}(f)$, we consider the compactification of $G$ as in Remark \ref{rmk_topentropy}. We use the variational principle $h_{\rm top}(f) = \sup_{\mu}h_{\mu}(f)$ where the supremum is taken over all the invariant Borel probability measures on $\overline{G}$. For each such measure $\mu$, it is either supported on the point at the ends, or on a compact subset of $G$. Since there is no entropy at the end point of $G$, we are only interested in the case where $\mu$ is supported on a compact subset of $G$. In this case, $\sup_{\mu}h_{\mu}(f) =h_{\rm mme}(f) = \log \lambda(f)$. The conclusion follows from the variational principle and the first part of the lemma.



\end{proof}

\begin{lemma} \label{lem_6.1}
Let $f_1 \colon G_1 \to G_1$ and $f_2 \colon G_2 \to G_2$ be generalized endperiodic maps that are proper homotopy equivalent to each other. Let $\{\ell_i\}$ (resp. $\{\ell_i'\}$) be defined as in Definition \ref{l_idef} on $G_1$ (resp. $G_2$).
Then we have $h(\pi_1(G_1),f_1,\{\ell_i\}) = h(\pi_1(G_2),f_2,\{\ell_i'\})$.
\end{lemma}
\begin{proof}
Let $h \colon G_1 \to G_2$ be the homotopy equivalence. The key observation is that the sequence $\{\ell_i'\}$ and the pushforward sequence $\{h_*\ell_i\}$, {\color{black} viewed as functions on the conjugacy classes of $\pi_1(G_2)$,} are interleaving as defined in Remark \ref{rmk_2.11}. Therefore, $h(\pi_1(G_1),f_1,\{\ell_i\}) = h(\pi_1(G_2),f_2,\{\ell_i'\})$ by Remark \ref{rmk_2.11}.
\end{proof}

\begin{lemma} \label{lem_6.3}
Let $f_1 \colon G_1 \to G_1$ and $f_2 \colon G_2 \to G_2$ be generalized endperiodic relative train track maps that are proper homotopy equivalent to each other. Then $\lambda(f_1) = 0$ if and only if $\lambda(f_2) = 0$.
\end{lemma}
\begin{proof}
We show that $\lambda(f)=0$ if and only if (*)for any conjugacy $\mathfrak{c} \in \pi_1(G)$, any compact subgraph $G'$ of $G$, there exists $N \in \mathbb{N}$ such that for any $n \ge N$, we have
$f^n(c)$ equals the identity element in the fundamental group $\pi_1(G_1)$ where $G_1$ is the finite graph obtained by collapsing all edges in $G\setminus G'$. Since (*) is invariant under proper homotopy equivalence, the lemma follows.

If $\lambda(f) = 0$, then (*) follows from Theorem \ref{thm_1.2} (1). Conversely, suppose $\lambda(f) \neq 0$. Then let $G'$ be the strata with $\lambda(f)$, $\mathfrak{c}$ be a conjugacy class in $G'$. Then for any $N\in \mathbb{N}$, there exists $n \ge N$ such that $f^n(\mathfrak{c})$ not equal to identity as the growth rate of $\mathfrak{c}$ equals $\lambda(f)$.
\end{proof}

Now we give a proof of Theorem \ref{thm_1.4}.
\begin{proof}[Proof of Theorem \ref{thm_1.4}]
Let $f_1 \colon G_1 \to G_1$ and $f_2 \colon G_2 \to G_2$ be generalized endperiodic maps that are proper homotopic to each other. After applying Theorem \ref{thm_traintrack}, denote by $F_1 \colon G'_1 \to G'_1$ and $F_2 \colon G'_2 \to G'_2$ the relative train track representations of $f_1$ and $f_2$ respectively.
We note that $F_1$ and $F_2$ are proper homotopy equivalent to each other since $F_i$ is proper (or combinatorially bounded) homotopy equivalent to $f_i,i=1,2$ and $f_1$ is proper homotopy equivalent to $f_2$.

If $\lambda(F_1)=0$, then $\lambda(F_2) = 0$ by Lemma \ref{lem_6.3}.
Suppose that $\lambda(F_1)\ge 1$. Again by Lemma \ref{lem_6.3}, we also have $\lambda(F_2) \ge 1$.
By Lemma \ref{lem_6.1}, we have $h(\pi_1(G_1'),F_1,\{\ell_i\}) = h(\pi_1(G_2'),F_2,\{\ell'_i\})$.
On the other hand, by Lemma \ref{lem_2.13}, we have
$h(\pi_1(G_1),F_1,\{\ell_i\}) = \log\lambda(F_1)$ and $h(\pi_1(G_2),F_2,\{\ell_i'\}) = \log\lambda(F_2)$.

Therefore, $\lambda(F_1) = \lambda(F_2)$ in both cases. 
By Lemma \ref{lem_2.13}, we have $\lambda(F_1) = h_{\rm top}(F_1)$.
The last sentence of the theorem follows from Lemmas \ref{lem_2.13}, \ref{lem_6.1} and \ref{lem_2.130} if $\lambda(F_1)\ge 1$; the last sentence is trivial if $\lambda(F_1) = 0$. This completes the proof.
\end{proof}

\section{Further Questions}\label{sec_questions}
Comparing the example in Section \ref{ex2} and \cite[Example 4.13]{CCF21}, we see that the relative train track map we obtained is not necessarily the same as the train tracks obtained via the Handle-Miller theory developed in \cite{CCF21}. Even after doing further folds to turn it into train track, it remains different from the one in \cite{CCF21}.
\begin{question}
What is the relationship between the relative train track maps or train track maps we constructed and the ones given by the Handle-Miller theory?
\end{question}

\begin{question}
Can our Bestvina-Handel algorithm be extended to larger families of infinite graph maps? For instance those that are not necessarily attracting or repelling near the ends, but commutes with an endperiodic map near each end.
\end{question}

\bibliographystyle{siam}
\bibliography{refs}
\end{document}